\documentclass[a4paper,12pt]{article}

\usepackage{amsmath,amssymb,amsthm}
\usepackage{graphicx,color,hyperref}
\usepackage{rawfonts}

\input prepictex
\input pictex
\input postpictex

\newcommand{\zlom}{\penalty0\relax}

\newtheorem{theorem}{Theorem}
\newtheorem{lemma}[theorem]{Lemma}

\newtheorem{corollary}[theorem]{Corollary}

\def\email#1{\href{mailto:#1}{\texttt{#1}}}

\title{Connected even factors in the square of essentially 2-edge-connected graphs}
\author{Jan Ekstein\thanks{Department of Mathematics, Institute for Theoretical Computer Science, and European Centre of Excellence NTIS - New Technologies 
        for the Information Society, Faculty of Applied Sciences, University of West Bohemia, Pilsen, Technick\'a 8, 306 14 Plze\v n, Czech Republic, 
        e-mail: \email{ekstein@kma.zcu.cz}.}\and
        Baoyindureng Wu\thanks{Xinjiang University, Urumgi, Xinjiang, P.R.China, e-mail: \email{baoyin@xju.edu.cn}.}\and
        Liming Xiong\thanks{School of Mathematics and Statistics and Beijing Key Laboratory on MCAACI, Beijing Institute of Technology, Beijing, P.R.China,
        e-mail: \email{lmxiong@bit.edu.cn}.}}
\date{\today}

\begin{document}

\maketitle

\begin{abstract}
 An essentially $k$-edge connected graph $G$ is a connected graph such that deleting less than $k$ edges from $G$ cannot result in two nontrivial components.
 In this paper we prove that if an essentially 2-edge-connected graph $G$ satisfies that for any pair of leaves at distance 4 in $G$ there exists another
 leaf of $G$ that has distance 2 to one of them, then the square $G^2$ has a connected even factor with maximum degree at most 4. Moreover we show that,
 in general, the square of essentially 2-edge-connected graph does not contain a connected even factor with bounded maximum degree.


\end{abstract}

\section { Introduction}
We consider only finite undirected simple graphs. For terminology and notation not defined in this paper we refer to \cite{WES}. Let $G$ be a connected graph.
For vertices $x, y$ of $G$, let $N_G(x)$ denote the \emph{neighborhood} of $x$ in $G$, $d_G(x)=|N_G(x)|$ the \emph{degree} of $x$ in $G$, and
$\mbox{dist}_G(x,y)$ the \emph{distance} between $x, y$ in $G$. The \emph{square} of a graph $G$, denoted by $G^2$, is the graph with same vertex set as $G$
in which two vertices are adjacent if their distance in $G$ is at most 2. Thus $G\subseteq G^2$. There are several papers (e.g. see \cite{EKS}, \cite{EL},
\cite{FAU}, \cite{FLE}, \cite{GOU}, \cite{HEN}, \cite{CHA}, and \cite{CHI}) about hamiltonian properties in the square of a graph. This paper deals with
connected even factors which generalize some previous known results.

A \emph{factor} in a graph $G$ is a spanning subgraph of $G$. A \emph{connected even factor} in $G$ is a connected factor in $G$ in which every vertex has
positive even degree. A \emph{$[2, 2s]$-factor} of $G$ is a connected even factor of $G$ in which every vertex has  degree at most $2s$. There are some 
results about such kind of factors in terms of forbidden subgraphs (see \cite{Duan}, \cite{LiMingChu}, and \cite{Favaron}). Since a hamiltonian cycle is 
a $[2, 2s]$-factor with $s=1$, the minimum $s$ in a $[2, 2s]$-factor of a graph can be seen as a measure for how close a graph is to become hamiltonian. 
Furthermore we know from \cite{Under} that  it is NP-complete to determine whether the square of a graph is hamiltonian. Therefore the determination 
of minimum $s$ in a $[2, 2s]$-factor in the square of a graph is also NP-complete.

The result by Fleischner in \cite{FLE} concerning the existence of a hamiltonian cycle (a $[2, 2]$-factor) in the square of 2-connected graph is
well known. Recently, M\"{u}ttel and  Rautenbach in \cite{MUT} gave a shorter proof of this result.

\begin{theorem}\emph{\textbf{\cite{FLE}}}
 \label{Fleischner}
  If $G$ is a 2-connected graph and $v_1$ and $v_2$ are two distinct vertices of $G$, then $G^2$ contains a hamiltonian cycle $C$ such that both edges of $C$
  incident with $v_1$ and one edge of $C$ incident with $v_2$ belong to $G$. Furthermore, if $v_1$ and $v_2$ are neighbors in $C$, then these are three
  distinct edges.
\end{theorem}

Theorem \ref{Fleischner} was a base for proving the following theorem by Abderrezzak et al. in \cite{EL} using forbidden subgraphs.
The graph $S(H)$ is obtained from a graph $H$ by subdividing  each edge of $H$ exactly once.

\begin{theorem}\emph{\textbf{\cite{EL}}}
 \label{Abderrezzak}
  If $G$ is a connected graph such that every induced $S(K_{1,3})$ has at least three edges in a block of degree at most 2, then $G^{2}$ is hamiltonian.
\end{theorem}

Theorem \ref{Abderrezzak} was generalized by Ekstein et al. in \cite{EKS} for $[2, 2s]$-factors.

\begin{theorem}\emph{\textbf{\cite{EKS}}}
 \label{Ekstein}
  Let $s$ be a positive integer and $G$ be a connected graph such that every induced $S(K_{1, 2s+1})$ has at least three edges in a block of degree
  at most two. Then $G^2$ has a $[2, 2s]$-factor.
\end{theorem}

Let $G$ be a connected graph. Recall that a graph $G$ is \emph{essentially $k$-edge connected} if deleting less than $k$ edges from $G$ cannot result
in two nontrivial components. In this paper, we shall answer the question how it is for the existence of a $[2, 2s]$-factor in the square of a graph
with 2-edge (or essentially 2-edge)-connectivity instead of (vertex) connectivity of a graph.

A vertex of degree 1 is called {\sl a leaf}. A cut vertex $y$ is {\sl trivial} in $G$, if $y$ is not a cut vertex in $G-M$,
where $M$ is a set of all leaves adjacent  to $y$, otherwise is {\sl non-trivial}. If $M=\{x\}$ and the neighbor of $x$ is a trivial cut vertex of $G$,
then $x$ is called {\sl a bad leaf}. {\sl A trivial bridge} is a cut-edge of $G$ containing a leaf, otherwise is {\sl non-trivial}.
{\sl A bad bridge} is a trivial bridge of $G$ adjacent to a bad leaf. For illustration see Fig.~1.

\begin{figure}[ht]
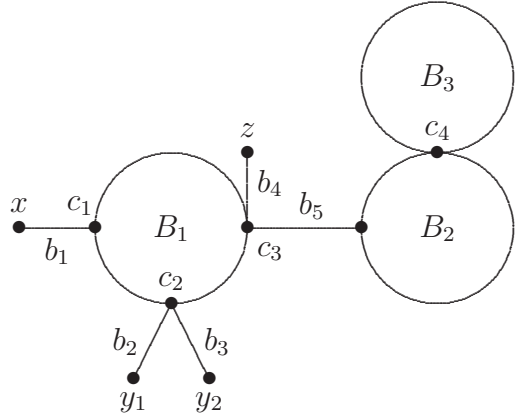

 \begin{center}
  \label{Figure1}
  $$\beginpicture
   \setcoordinatesystem units <1mm,1mm>
   \setplotarea x from 0 to 0, y from 0 to 30
   \put{$\bullet$} at -30  0
   \put{$x$} at -30 3
   \put{$\bullet$} at -20  0
   \put{$c_{1}$} at -22 3
   \plot -30 0 -20 0 /
   \put{$b_{1}$} at -25 -3
   \circulararc 360 degrees from -20 0 center at -10 0
   \put{$B_{1}$} at -10 0
   \put{$\bullet$} at -10 -10
   \put{$c_{2}$} at -10 -7
   \put{$\bullet$} at -15 -20
   \put{$y_{1}$} at -15 -23
   \put{$\bullet$} at  -5 -20
   \put{$y_{2}$} at -5 -23
   \plot -15 -20 -10 -10 -5 -20 /
   \put{$b_{2}$} at -16 -15
   \put{$b_{3}$} at  -4 -15
   \put{$\bullet$} at  0  0
   \put{$c_{3}$} at 3 -3
   \put{$\bullet$} at  0 10
   \put{$z$} at 0 13
   \put{$\bullet$} at 15  0
   \plot 15 0 0 0 0 10 /
   \put{$b_{4}$} at 3 6
   \put{$b_{5}$} at 8.5 3
   \circulararc 360 degrees from 15 0 center at 25 0
   \put{$B_{2}$} at 25 0
   \put{$\bullet$} at 25 10
   \put{$c_{4}$} at 25 13
   \circulararc 360 degrees from 25 10 center at 25 20
   \put{$B_{3}$} at 25 20
  \endpicture$$
  \caption{In this graph, $c_{1}, c_{2}$ are trivial cut vertices, $c_{3}, c_{4}$ are non-trivial cut vertices, $x$ is a bad leaf, $y_{1}, y_{2}, z$ are
  leaves, $b_{1}$ is a bad bridge, $b_{2}, b_{3}, b_{4}$ are trivial bridges, $b_{5}$ is a non-trivial bridge, and $B_{1}, B_{2}, B_{3}$ are cyclic blocks.}
 \end{center}
\end{figure}

\begin{figure}[ht]
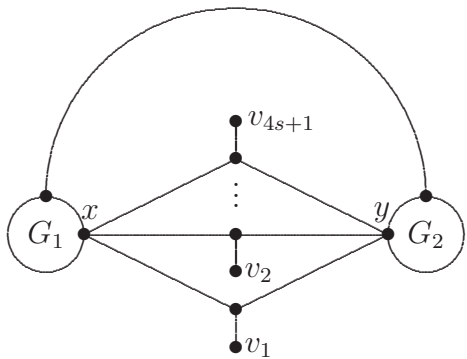

  \begin{center}
  $$\beginpicture
    \setcoordinatesystem units <1mm,1mm>
    \setplotarea x from 0 to 0, y from 0 to 30
    \put{$\bullet$} at -20  0
    \put{$\bullet$} at  20  0
    \circulararc 360 degrees from -20 0 center at -25 0
    \put{$G_{1}$} at -25 0
    \circulararc 360 degrees from 20 0 center at 25 0
    \put{$G_{2}$} at 25 0
    \put{$\bullet$} at 0 -15
    \put{$v_{1}$} at   3 -15
    \put{$\bullet$} at 0 -10
    \put{$\bullet$} at 0  -5
    \put{$v_{2}$} at   3  -5
    \put{$\bullet$} at 0   0
    \plot 0 -15 0 -10 -20 0 0 0 0 -5 0 0 20 0 0 -10 /
    \put{$\vdots$} at 0  6.25
    \put{$\bullet$} at 0  10 \put{$x$} at -19.2 3 \put{$y$} at 19.2 3
    \put{$\bullet$} at -25 5
    \put{$\bullet$} at 25 5
    \circulararc -180 degrees from -25  5 center at 0 5
    \put{$v_{4s+1}$} at 6 15
    \put{$\bullet$} at 0  15
    \plot -20 0 0 10 0 15 0 10 20 0 /
   \endpicture$$\label{Figure2}
   \caption{Essentially 2-edge connected graphs $G$ such that their square contains no $[2,2s]$-factor, where $G_{1}$ and $G_{2}$ are
            any essentially 2-edge connected graphs.}
  \end{center}
 \end{figure}

 Firstly, we look at the graph in Fig.~2, from which one may see the following result.
\begin{theorem}
 \label{Main-}
 For any fixed positive integer $s$, there exists an infinite class of essentially 2-edge-connected graphs $G$ such that $G^{2}$ has no $[2, 2s]$-factor, even
 if the resulting graph obtained from $G$ by deleting its all leaves is 2-connected.
\end{theorem}

\begin{proof}
 Note that the graph $G$ in Fig.~2 is an essentially 2-edge-connected graph. Since every leaf $v_i$ of $G$ has degree exactly 3 in $G^{2}$, at least one
 edge of $v_ix, v_iy$ have to be used in any possible $[2,4]$-factor of $G^2$. Therefore, $G^2$ has no $[2, 2s]$-factor since $G$ has  $4s+1$ such leaves.
\end{proof}

On the other hand, we may show the following result, which is the main result of this paper.

\begin{theorem}
 \label{Main++}
  Let $G$ be a connected graph without non-trivial bridges and without any two bad leaves at distance exactly 4. Then $G^{2}$ has a $[2,4]$-factor.
\end{theorem}

The following corollaries are immediate consequences of Theorem~\ref{Main++}.

\begin{corollary}
 \label{Cor1}
  If $G$ is a 2-edge connected graph, then $G^2$ contains a $[2,4]$-factor.
\end{corollary}

\begin{corollary}
 \label{Cor2}
  If $G$ is an essentially 2-edge connected graph without bad leaves, then $G^2$ contains a $[2,4]$-factor.
\end{corollary}

\begin{corollary}
 \label{Main+}
 Let $G$ be a connected graph without non-trivial bridges. If any two bad leaves have distance at least 5 in $G$, then $G^{2}$ has a $[2,4]$-factor.
\end{corollary}

 Note that the graph in Fig.~2 also shows that the distance 5 in Corollary~\ref{Main+} can not be replaced by distance 4.

\section{A \emph{Useful lemma}}

Before presenting this lemma, we need some additional notation. {\sl Block graph} of a graph $G$, denoted by $BC(G)$, is the graph whose vertex set
consists of all blocks and cut vertices of $G$, and two vertices are adjacent in $BC(G)$ if one of them is a block of $G$ and the second one is its vertex.
It is easy to see that $BC(G)$ is a tree for a connected graph $G$. Note that for any tree, we may choose any vertex as its root. Hence without loss of
generality, we may assume that $B_1, \ldots, B_t$ be all blocks of $G$ such that $B_1$ corresponds to the root of $BC(G)$. For a cut-vertex $v$ of $G$,
{\sl the parent block} of $v$ is the block containing $v$ and its corresponding vertex  in $BC(G)$ has the smallest distance to the root of $BC(G)$.
The remaining blocks containing $v$ are called  {\sl children blocks} of $v$ with respect to the root of $BC(G)$.

The following lemma, we call it a \emph{Useful lemma}, is a key for the proof of  our main result (Theorem ~\ref{Main++}).

\begin{lemma}(Useful lemma)
 \label{Main}
  Let $G$ be a connected graph without non-trivial bridges and without bad leaves (except $K_{1,2}, K_{1,3}$) and $u$ be a vertex of $G$ that is neither
  a cut vertex nor a leaf (if any).

  \noindent
  Then $G^{2}$ has a $[2,4]$-factor $F$ such that
  \begin{itemize}
   \item [a)] $d_F(x)=2$ for any vertex $x$ that is not a cut vertex of $G$;
   \item [b)] both edges of $F$ incident with $u$ belong to $G$;
   \item [c)] for each cut vertex $y$ of $G$ it holds that $d_F(y)=4$ and at least two edges of $F$ incident with $y$ belong to $G$,
              moreover if $y$ is a trivial cut vertex, then these two edges are trivial bridges;
   \item [d)] for any cut vertex $y$ of $G$, the two edges incident with $u$ in $F$ are distinct from the two edges incident with $y$ in $F$
              as specified in (c);
   \item [e)] for any two cut vertices $y_1$ and $y_2$ of $G$, the two edges of $F$ incident with $y_1$ as specified in $(c$) are distinct from those
              with $y_2$ as specified in~(c).
  \end{itemize}
\end{lemma}

\begin{proof}
 If $G$ is $K_{1,s}$, for $s\geq4$, then $G^2$ is a complete graph and the result is obvious. Now we assume that $G$ contains at least one cyclic block and 
 $G'=G-M$, where $M$ is a set of all leaves adjacent with all trivial cut vertices of $G$.

 Let $\mathbb{O}=B_{1}, B_{2}, ..., B_{k}$ be an ordering of all blocks of $G'$ such that either $u\in V(B_{1})$, if any, or we choose arbitrary cyclic block
 as $B_{1}$, satisfying the following properties:

 - for any cut vertex $v$ of $G'$, all children blocks of $v$ with respect to the root $r$ of $BC(G')$ corresponding to $B_{1}$ appear consecutively
   in $\mathbb{O}$ such that bridges containing $v$ are in $\mathbb{O}$ before cyclic blocks containing $v$;

 - $\mbox{dist}_{BC(G')}(r, v_{i})<\mbox{dist}_{BC(G')}(r, v_{j})$ implies $i<j$, where $v_{i}, v_{j}$ are vertices of $BC(G')$ corresponding
   to $B_{i}, B_{j}$, respectively.

 \medskip

 Then $G'$ is a connected graph without non-trivial bridges and without bad leaves and we prove by induction on $k$ that $(G')^{2}$ contains
 a $[2,4]$-factor $F'$ such that

 \begin{itemize}
  \item[1)] $d_{F'}(x)=2$ for any vertex $x$ that is not a cut vertex of $G$;

  \item[2)] both edges of $F'$ incident with $u$, if any, belong to $B_{1}$;

  \item[3)] for each cut-vertex $y$ of $G'$, it holds that $d_{F'}(y)=4$ and at least two edges of $F'$ incident with $y$ belong to $G'$. Moreover,

   \begin{itemize}
    \item  if $y$ belongs to exactly two blocks of $G'$, then at least two edges of $F'$ incident with $y$ are edges from  the children block of $y$
             with respect to $r$ (the root of $BC(G')$ corresponding to $B_{1}$);

    \item  if $y$ belongs to more than two blocks of $G'$, then at least two edges of $F'$ incident with $y$ are edges from two different children blocks
             of $y$ with respect to $r$.
   \end{itemize}
 \end{itemize}

 For $k=1$, $G'=B_{1}$ and $(G')^{2}$ even has a hamiltonian cycle $C$ such that both edges of $F'$ incident with $u$, if any, belong to $B_{1}$
 by Theorem~\ref{Fleischner}.

 Let $k>1$ and assume that Lemma~\ref{Main} is true for all integers less than $k$. By the definition of $G'$ and $\mathbb{O}$, $B_{k}$ is an end cyclic block of $G'$
 and let $v_{0}$ be the cut vertex of $G'$ with $v_{0}\in V(B_{k})$.

 If $B_{k-1}=v_{0}l$ (i.e. $B_{k-1}$ is a bridge) and $B_{k-1}, B_{k}$ are only children blocks of $v_{0}$ with respect to $r$, then
 we set $G_{1}=G' - \{V(B_{k})\cup\{l\}\setminus\{v_{0}\}\}$, otherwise we set $G_{2}=G' - \{V(B_{k})\setminus\{v_{0}\}\}$. Hence $G_{1}, G_{2}$ are
 connected graphs without non-trivial bridges and without bad leaves and have $k-2, k-1$ blocks, respectively. Hence by the induction hypothesis,
 $(G_{1})^{2}, (G_{2})^{2}$ have a $[2,4]$-factor $F_{1}, F_{2}$ with properties 1), 2), and~3), respectively.

 By Theorem~\ref{Fleischner}, there is a Hamiltonian cycle $C$ in $(B_{k})^{2}$ such that two edges $f_{1}, f_{2}$ of $C$ incident with $v_{0}$ belong
 to $B_{k}$ and thus belong to $G'$.

 \medskip

 \noindent
 \emph{Case 1:} $G_{1}$ exists.

  Let $f_{1}=v_{0}v_{k}$. Then $F'=(F_{1}\cup C)\cup\{v_{0}l,v_{k}l\}-\{f_{1}\}$ is the $[2,4]$-factor of $(G')^{2}$ with properties 1), 2), and 3).

 \medskip

 \noindent
 \emph{Case 2:} $G_{1}$ does not exist and $v_{0}$ is not a cut vertex in $G_{2}$.

  Hence $v_{0}$ belongs to exactly two blocks of $G'$ and $F'=F_{2}\cup C$ is the $[2,4]$-factor of $(G')^{2}$ with properties 1), 2), and 3).

 \medskip

 \noindent
 \emph{Case 3:} $G_{1}$ does not exist and $v_{0}$ is a cut vertex in $G_{2}$.

 Let $f_{1}=v_{0}v_{k}$. We consider two possibilities depending on the property~3).

 If exactly two blocks of $G_{2}$ contain $v_{0}$, then by the induction hypothesis $d_{G_{2}}(v_{0})=4$ and there are two edges of $F_{2}$ incident
 with $v_{0}$ from a children block $B_{k-1}$ of $v_{0}$. (Note that $B_{k-1}$ is a cyclic block, since $G_{1}$ does not exist.) Let $e_{k-1}=v_{0}v_{k-1}$
 be  such an edge of $F_{2}$. Since $\mbox{dist}_{G'}(v_{k-1},v_{k})=2$, the edge $v_{k-1}v_{k}$ is an edge of $(G_{2})^{2}$.
 Thus $F'=(F_{2} \cup C)\cup\{v_{k-1}v_{k}\}-\{e_{k-1},f_{1}\}$ is the $[2,4]$-factor of $(G')^{2}$ with properties 1), 2), and 3).

 If there are more than two blocks of $G_{2}$ containing $v_{0}$, then by the~induction hypothesis $d_{G_{2}}(v_{0})=4$ and there are two edges
 $e_{k-2},e_{k-1}$ of $F_{2}$ incident with $v_{0}$ in $B_{k-2},B_{k-1}$, respectively. Note that it could be $B_{k-2}=e_{k-2}$ or $B_{k-1}=e_{k-1}$. 
 Let $e_{k-2}=v_{0}v_{k-2}$. Since $\mbox{dist}_{G'}(v_{k-2},v_{k})=2$, the edge $v_{k-2}v_{k}$ is an edge of $(G')^{2}$. Thus 
 $F'=(F_{2}\cup C)\cup\{v_{k-2}v_{k}\}-\{e_{k-2},f_{1}\}$ is the $[2,4]$-factor of $(G')^{2}$ with properties 1), 2), and 3).

 \medskip

 Now we extend $F'$ to a $[2,4]$-factor $F$ in $G^{2}$ with required properties. Note that the properties 1), 2), and 3) imply the
 properties a)-e) in Lemma~\ref{Main}.

 Let $u_{1},u_{2},...,u_{t}$ be all trivial cut vertices of $G$ and $l_{i}^{1},l_{i}^{2},...,l_{i}^{s_{i}}$ be all leaves incident with $u_{i}$,
 for $i=1,2,...,t$. Note that $s_{i}\geq2$, otherwise we have a bad bridge in $G$, a contradiction. For $i=1,2,...,t$, let
 $C_{i}=u_{i}l_{i}^{1}l_{i}^{2}...l_{i}^{s_{i}}u_{i}$ be cycles in $G^{2}$ and $C'=\cup_{j=1}^{t}C_{j}$. Since $d_{F'}(u_{i})=2$ and
 $u_{i}l_{i}^{1}, l_{i}^{s_{i}}u_{i}$ are edges from $G$, $F=F'\cup C'$ is the $[2,4]$-factor of $G^{2}$ with properties a)-e).
\end{proof}

Note that clearly the square of $K_{1,2}$, $K_{1,3}$ is hamiltonian but there is no $[2,4]$-factor with a vertex of degree 4 in the square of $K_{1,2}$,
$K_{1,3}$, respectively.

\section{Proof of Theorem~\ref{Main++}}
In this section we prove   Theorem~\ref{Main++}.


\begin{proof}
 Firstly if $G$ is $K_{1,2}$ or $K_{1,3}$, then clearly $G^{2}$ is even hamiltonian.

 Now let $X$ be a set of all bad leaves of $G$ and $G'=G-X$. For $x_{i}\in X$, we denote $y_{x_{i}}$ or only $y_{i}$ its unique neighbor in $G$.
 By Lemma~\ref{Main}, there is a [2,4]-factor $F'$ of $(G')^2$ with properties a)-e). Note that $d_{F'}(y_{i})=2$ for each $y_{i}$.

 By the definition, any two bad leaves have a distance at least 3. Let $X_{0}\subseteq X$ be the set of all bad leaves that has a bad leaf
 at the distance exactly~3 in $G$. Then, for all $x_{i}\in X_{0}$, corresponding $y_{i}$'s induce a subgraph of $G'$ in which all components (denoted by
 $H_{1},H_{2},...,H_{s}$) are complete graphs, otherwise we have in $G$ two bad leaves at distance 4, a contradiction.

 Let $V(H_{i})=\{y_{i,1},y_{i,2},...,y_{i,t_{i}}\}$, $t_{i}\geq 2$ for $i=1,2,...,s$. Then we set

 $$M_{i}=\bigcup_{j=1}^{t_i-1}\{x_{i,j}y_{i,j+1}, x_{i,j+1}y_{i,j}\}~\bigcup~\{x_{i,1}y_{i,1}, x_{i,t_i}y_{i,t_i}\}.$$

 All bad leaves of $X\setminus X_{0}$ are pairwise at distance at least 5 and we divide them into the following three disjoint classes by the following way
 (see Fig.~\ref{Figure3} for illustration):

  \begin{figure}[ht]
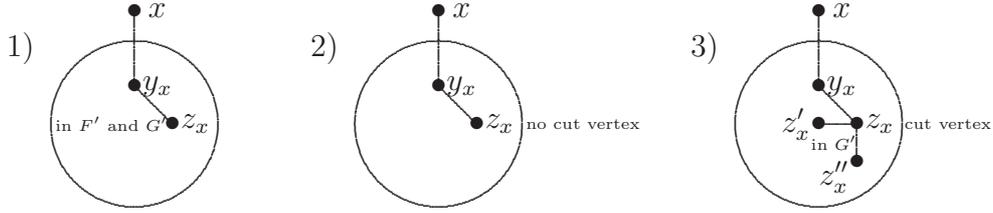

  \begin{center}
   \label{Figure3}
   $$\beginpicture
    \setcoordinatesystem units <1mm,1mm>
    \setplotarea x from -30 to 0, y from -10 to 10
    \put{$1)$} at -65 10
    \put{$2)$} at -25 10
    \put{$3)$} at 25 10
    \circulararc 360 degrees from -39 0 center at -50 0
    \circulararc 360 degrees from -21 0 center at -10 0
    \circulararc 360 degrees from  29 0 center at 40 0
    \put{$y_{x}$} at -47 5
    \put{$\bullet$} at -50 5
    \put{$x$} at -47 15
    \put{$\bullet$} at -50 15
    \put{$z_{x}$} at -42 0
    \put{$\bullet$} at -45 0
    \put{\tiny {in $F'$ and $G'$}} at -53 0
    \put{$y_{x}$} at -7 5
    \put{$\bullet$} at -10 5
    \put{$x$} at -7 15
    \put{$\bullet$} at -10 15
    \put{$z_{x}$} at -2 0
    \put{$\bullet$} at -5 0
    \put{\tiny{no cut vertex}} at 9 0
    \put{$y_{x}$} at 43 5
    \put{$\bullet$} at 40 5
    \put{$x$} at 43 15
    \put{$\bullet$} at 40 15
    \put{$z_{x}$} at 48 0
    \put{$\bullet$} at 45 0
    \put{\tiny{cut vertex}} at 57 0
    \put{$z'_{x}$} at 37 0
    \put{$\bullet$} at 40 0
    \put{$z''_{x}$} at 42 -7
    \put{$\bullet$} at 45 -5
    \put{\tiny{in $G'$}} at 42 -2.5
    \plot -45 0 -50 5 -50 15 /
    \plot -5 0 -10 5 -10 15 /
    \plot 45 0 40 5 40 15 /
    \plot 40 0 45 0 45 -5 /
   \endpicture$$
   \caption{Three cases in an ordering of all bad leaves of $X\setminus X_0$ in $G$.}
  \end{center}
 \end{figure}

 \begin{itemize}
  \item [1)] Let $X_1$ be the set of all vertices $x\in X\setminus X_0$ such that there exists a vertex $z_x$ with $y_xz_x\in E(F')\cap E(G')$;
  \item [2)] Let $X_2$ be the set of all vertices $x\in X\setminus (X_0\cup X_1)$ such that there exists $z_x$, which is not a cut vertex of $G'$, 
             with $y_xz_x\in E(G')$ (and $y_xz_x\in E(F')$);
  \item [3)] Let $X_3$ be the set of all vertices $x\in X\setminus (X_0\cup X_1\cup X_2)$ (it means that there exists only a cut vertex $z_x$ of $G'$
             with $y_xz_x\in E(G')$ (and $y_xz_x\in E(F')$).
\end{itemize}

 Note that by Lemma~\ref{Main} we have
  \begin{itemize}
   \item  $d_{F'}(z_x)=2$ \ for $x\in X_2$;
   \item  $d_{F'}(z_x)=4$  and at least two edges incident with $z_x$ (namely $z_xz'_x, z_xz''_x)$ are in $E(F')\cap E(G')$ for $x\in X_3.$
  \end{itemize}
 Now set
 $$E_0=\bigcup _{i=1}^{s}M_{i},~~~ E_{1}=\bigcup_{x\in X_1}\{xy_x, xz_x\},~~~ E'_{1}=\bigcup_{x\in X_1}\{y_xz_x\},$$
 $$E_{2}=\bigcup_{x\in X_2}\{xy_x, xz_x, y_xz_x\},$$
 $$E_{3}=\bigcup_{x\in X_3}\{xy_x, xz_{x}, y_{x}z'_{x}\},~~~
   E'_{3}=\bigcup_{x\in X_3}\{z_{x}z'_{x}\}.$$

 For all $x$, $z_x$'s are different, otherwise if $z_{x}=z_{x'}$, for $x\neq x'$, then $xy_xz_x(=z_{x'})y_{x'}x'$ is
 a path of length 4 in $G$ joining two bad leaves, a contradiction. Similarly, none of $z_{x}$'s is a neighbor of a bad leaf in $G$.

 Possibly, $z_{x_{i_{1}}}z_{x_{i_{2}}}...z_{x_{i_{k}}}$ is a path in $F'$ for $\{x_{i_{1}},x_{i_{2}},...,x_{i_{k}}\}\subseteq X_3$.
 In order to have different edges in $E_{3}$ and $E'_{3}$ we set $z'_{x_j}=z_{x_{j+1}}$, for $j= i_{1},i_{2},...,i_{k-1}$, and
 $z'_{x_{i_{k}}}$ as arbitrary neighbor of $z_{x_{i_{k}}}$ in $F'$ and in $G$ different from $z_{x_{i_{k-1}}}$. Note that by 3) and
 Lemma~\ref{Main} such a vertex exists and could be some $z_{x_j}$, for $j\in \{i_{1},i_{2},...,i_{k-2}\}$.

 Hence we conclude that $F=F'\cup (E_{0}\cup E_{1}\cup E_{2}\cup E_{3})-(E'_{1}\cup E'_{3})$ is a~[2,4]-factor of $G^2$.
\end{proof}

\section{Conclusion}

Now we could answer the question from Introduction. By Theorem~\ref{Fleischner} we know that the square of 2-connected graph has a $[2,2s]$-factor
for $s=1$. In this paper we proved that the square of 2-edge-connected graph has a $[2,2s]$-factor for $s=2$ (Corollary~\ref{Cor1})
and that the square of essentially 2-edge-connected graph without bad leaves has a $[2,2s]$-factor also for $s=2$ (Corollary~\ref{Cor2}).
In general, there exist essentially 2-edge-connected graphs whose square have no $[2,2s]$-factor for every $s$.  This example of $G$ even exists under an
additional condition  that the graph obtained from $G$ by deleting all leaves is 2-connected (Theorem~\ref{Main-}).

\medskip

\noindent {\bf Acknowledgements}.

This work was supported by the European Regional Development Fund (ERDF), project NTIS - New Technologies for Information Society,
European Centre of Excellence, CZ.1.05/1.1.00/02.0090.

The first author was supported by project GA14-19503S of the Grant Agency of the Czech Republic.

The second author was supported by NSFC (No.11161046) and by Xinjiang Talent Youth Project (No.2013721012).

The third author was supported by NSFC (No.11471037 and No.11671037).

\end{document}